\documentclass[11pt,a4paper]{article}

\usepackage{amsthm}
\usepackage{amsmath}
\usepackage{epsfig}
\usepackage{amssymb}

\theoremstyle{plain}
\newtheorem{thm}{Theorem}
\newtheorem{prp}[thm]{Proposition}
\newtheorem{lem}[thm]{Lemma}
\newtheorem{cor}[thm]{Corollary}

\newtheorem{defi}[thm]{Definition}

\newcommand{\R}{\mathbb{R}} 
\newcommand{\N}{\mathbb{N}} 

\begin{document}
\author{David Barbato and Francesco Morandin}
\title{Positive and non-positive solutions\\ for an inviscid dyadic model.\\ Well-posedness and regularity.}
\maketitle
\begin{abstract}
We improve regolarity and uniqueness results from the literature for
the inviscid dyadic model.  We show that positive dyadic is globally
well-posed for every rate of growth $\beta$ of the scaling
coefficients $k_n=2^{\beta n}$. Some regularity results are proved for
positive solutions, namely $\sup_n
n^{-\alpha}k_n^{\frac13}X_n(t)<\infty$ for a.e.\ $t$ and $\sup_n
k_n^{\frac13-\frac1{3\beta}}X_n(t)\leq Ct^{-1/3}$ for all $t$.
Moreover it is shown that under very general hypothesis, solutions
become positive after a finite time.
\end{abstract}

\section*{Introduction}
Well-posedness and regularity for Navier-Stokes and Euler equations
represent a major point of interest in mathematics. The study of
estimates of the nonlinear term $(u\cdot\triangledown) u$ in
particular is important, since this term is associated with the
so-called cascade of energy from lower to higher modes.

A very rough idea of this phenomenon is the following. Fix a time $t$
and decompose the velocity $u$ on the frequencies $u=\sum_k u_k e_k$,
where $e_k$ are ``wave functions'' (for example $e_k=\sin(\langle
k,\cdot\rangle)$ or $e_k=e^{i\langle k,\cdot\rangle}$) and $u_k$ are the
corresponding coefficients (with $\sum_k |u_k|^2=\|u\|_{L^2}^2
$). Then the regularity of $u$ can be associated to how fast the
coefficients $u_k$ go to zero as $|k|$ tends to infinity. More
precisely, $u$ has $N$-th derivative in $L^2$ if $\sum_k(|k|^N
|u_k|)^2<\infty$. The bilinear term $(u\cdot\triangledown) u$ acts on
the dynamics of the coefficients $u_k$ by mixing different components,
in that if $h$ and $k$ are two active frequencies ($u_h,u_k\neq 0$),
then the term $(u\cdot\triangledown) u$ will activate the component
$h+k$ and so on, activating higher and higher frequencies in a
phenomenon called \emph{energy cascade}~\cite{Bif}.

In this paper we study the inviscid dyadic model, which is a
shell-type model of Euler equations which was early introduced
in~\cite{DesNov1974} and then again in recent times
in~\cite{KatPav} and~\cite{Wal}. This model, represented in
equations~\eqref{e:dyadic_unviscous} below, exploits some of the
properties of Euler equations (see among the others~\cite{Bif}
and~\cite{EggGro}) on a much simpler structure. Informal derivations
of this model from Euler equations are given in~\cite{DesNov1974},
\cite{KatPav} and~\cite{Che}; from an intuitive point
of view, one should imagine that the variable $X_n$
in~\eqref{e:dyadic_unviscous} represents a global coefficient for all
components $u_k$ with $|k|$ of order $2^n$. These equations, like
Euler's, are homogeneous of degree 2, they are formally conservative
and moreover they show the energy cascade phenomenon, which in this
setting is very clearly understood~\cite{BarFlaMor2011TAMS}.

To collect previous results on the dyadic model, one should
immediately distinguish whether the initial condition has all positive
components or not, since the dynamics of energy cascade is strongly
dependent on the sign of $X_n$. If all $X_n$'s are positive, energy
moves from lower modes to higher ones.  If all $X_n$'s are negative,
energy moves from higher modes to lower ones.
In~\cite{BarFlaMor2011TAMS} and~\cite{CheFriPav2010} it is shown that
in the case of positive components, energy moves to higher modes
faster and faster, in such a way that a positive fraction of energy
gets lost ``at infinity'' in finite time and one can show that in this
case the energy $\|X(t)\|_{l^2}^2$ goes to zero like $t^{-2}$. On the
other hand in~\cite{CheFriPav2010} it is proved that if a positive
forcing term is included in the model (which puts energy into the
first component), the energy converges exponentially fast to a fixed
value (corresponding to the stationary solution). In both cases the
model is not really conservative in the end, since energy moves to
infinity and there it disappears: this phenomenon is called anomalous
dissipation.

On the other hand, if all components are negative, the opposite
situation can occur: energy can enter from ``infinity'' into the
system. In~\cite{BarFlaMor2011TAMS} explicit solutions are costructed
in which there is an anomalous increase of energy, immediately
yielding non-uniqueness of solutions for the negative dyadic.

In this paper we prove that the positive dyadic is globally
well-posed, extending the uniqueness result
in~\cite{BarFlaMor2010CRAS} to arbitrary rate of growth of the
coefficients $k_n$ in system~\eqref{e:dyadic_unviscous}. This means in
particular that the escape of energy at infinity does never preclude
well-posedness; on the contrary, in the negative dyadic, the input of
energy from infinity immediately destroys well-posedness.

It is interesting to confront this with the stochastic dyadic model
introduced in~\cite{BarFlaMor2010PAMS}, where the distinction between
positive and negative solutions is meaningless, since noise causes
infinite sign changes in every time interval.
In~\cite{BarFlaMor2011AAP} and~\cite{BarFlaMor2010PAMS} it is proved
that there is escape of energy at infinity and (weak) uniqueness of
solutions, so also in this case energy cannot enter from infinity and
the problem is well-posed.

Section~\ref{s:negative} deals with the connection between negative and
positive dyadic. It is shown that under minimal hypothesis all
components become positive in finite time and stay positive
forever. Based on this fact all the rest of the paper is restricted to
positive initial conditions.

Sections~\ref{s:regularity}, \ref{sec:uniqueness}
and~\ref{s:invariantregion} deal with uniqueness and regularity and
are very much interlinked. Uniqueness was already proved
in~\cite{BarFlaMor2010CRAS} for $k_n=2^{\beta n}$ and $\beta\leq1$ and
in~\cite{BarMorRom} for any $\beta$ in a class of regular enough
solutions. It is now extended by Theorem~\ref{thm_unicita} to
arbitrary $\beta$ and $l^2$ initial condition.

As already stated we are interpreting the components of the solution
as something similar to Fourier coefficients, so regularity means
smallness of components $X_n$ for $n$ large. The vague idea is that
$X_n$ tends to zero as $k_n^{-1/3}$. The first results in the
literature are of lack of regularity. In~\cite{Che}, \cite{KatPav},
\cite{CheFriPav2010}, \cite{KisZla} and~\cite{BarFlaMor2011TAMS} it is
shown that if the initial condition is in $l^2$, then all solutions
are Leray-Hopf and nontheless a blow-up occurs in finite time, in the
sense that for all $\epsilon>0$ the quantities
\begin{align*}
&\sup_n k_n^{\frac{1}{3}+\epsilon} X_n(t),&
&\sum_n k_n^{\frac{1}{3}} X_n(t),&
&\sup_n \int_0^t k_n^{(1+\epsilon)} X^3_n(s)ds,
\end{align*}
become infinite in finite time, even if they were finite for $t=0$.

On the other hand one first important regularity result can be found
again in~\cite{CheFriPav2010}, where for $\beta<3$, the authors prove
that for all $\epsilon>0$,
\[
\sup_n \int_0^t (k_n^{1/3-\epsilon}X_n(s))^2 ds<+\infty.
\]

Our main results on regularity are Theorem~\ref{thm_reg1},
Lemma~\ref{lem_intX3_bound}, Theorem~\ref{thm_inv_reg} and
Theorem~\ref{cor_limit_X_3} which, through some corollaries, imply that
\begin{enumerate}
\item $\displaystyle\sup_{n}n^{-\alpha} k_n^{\frac13} X_n(t)<\infty$ for all
  $\alpha>\frac13$ and for a.e.\ $t>0$;
\item $\displaystyle\sum_{n}n^{-\alpha} k_n^{\frac13} X_n(t)<\infty$ for all
  $\alpha>\frac43$ and for a.e.\ $t>0$;
\item $\displaystyle\sup_n \int_0^t n^{-1}k_n X^3_n(s)ds<\infty$  for all $t>0$;
\end{enumerate}
and moreover
\begin{enumerate}\setcounter{enumi}{3}
\item $\displaystyle\sup_{n} k_n^{\frac{1}{3}-\frac1{3\beta}} X_n(t)\leq Ct^{-1/3}$  for
  all $t>0$.
\end{enumerate}

\section{Model}
A very natural space for the dynamics of the dyadic is $H:=l^2(\R)$,
the Hilbert space of square-summable sequences with the usual norm
which we will denote simply by $\| \cdot \|$.

Let $\beta>0$ and $x=(x_n)_{n\geq1}\in H$. Consider the following
Cauchy problem
\begin{equation}\label{e:dyadic_unviscous}
  \begin{cases}
    \dot X_n
      = k_{n-1} X_{n-1}^2
        -k_n X_n X_{n+1},\\
    X_n(0)=x_n,
  \end{cases}
  \qquad n\geq1\quad t\geq0
\end{equation}
where $X_0=0$, $k_0=0$ and $k_n = 2^{\beta n}$ for $n\geq1$.

\begin{defi}
A weak solution is a sequence $X=(X_n)_{n\geq1}$ of differentiable
functions on all $[0,\infty)$, satisfying~\eqref{e:dyadic_unviscous}.

A finite energy solution is a weak solution such that $X(t)$ is in $H$
for all $t\geq0$.
\end{defi}

The following proposition (whose proof is immediate) shows that
without loss of generality we can suppose that the initial condition
$x$ has positive first component $x_1>0$ and arbitrarily small norm.

\begin{prp}\label{prp_x1_pos}
Suppose $(X_n)_{n\geq1}$ is a weak solution
of~\eqref{e:dyadic_unviscous} with initial condition $x\neq0$. We
denote by $\bar{n}$ its first non-zero index: $\bar
n:=\min\{n\geq1:x_n\neq0\}$, so that $X_n\equiv0$ for $n<\bar{n}$,
while $X_{\bar n}(t)\neq0$ for all $t>0$.  Let $\alpha>0$ and let, for
$n\geq1$,
\begin{gather} 
\label{eq_cambio_sca_1}
Y_n(t):=
\begin{cases}
 X_n(t) & n\neq \bar{n}\\
-X_n(t) & n=\bar{n}
\end{cases}, \qquad\qquad
Z_n(t):=X_{n+\bar{n}-1}\left(\frac{t}{k_{\bar{n}-1}}\right),\\
\label{eq_cambio_sca_2}
W_n(t):=\alpha X_n(\alpha t).
\end{gather}
Then all of the above, $(Y_n)_{n\geq1}$,
$(Z_n)_{n\geq1}$ and $(W_n)_{n\geq1}$, are weak
solutions of \eqref{e:dyadic_unviscous} each one with its own initial
condition.
\end{prp}

\begin{defi}
A Leray-Hopf solution is a finite energy solution such that
$\|X(t)\|$ is a non-increasing function of $t$.
\end{defi}
We also introduce the notation for the finite-size blocks energy: for
all $n\geq1$ let
\begin{equation}
E_n(t):=\sum_{i\leq n}X_i^2(t).
\end{equation}
A direct computation shows that 
\begin{equation}\label{eq:energy_least_comp_nonincr}
E_n'=-2k_nX_n^2X_{n+1}, 
\end{equation}
so we can
study the variations of energy by looking to the sign of components.

\begin{prp}\label{thm:positivity_mantained}
Let $X$ be a weak solution of \eqref{e:dyadic_unviscous} and
$t_0\geq0$. If $X_n(t_0)>0$ then $X_n(t)>0$ for all $t\geq t_0$. If
$X_n(t_0)\geq 0$ then $X_n(t)\geq 0$ for all $t\geq t_0$.
\end{prp}
\begin{proof}
Simply apply the variation of constants formula
to~\eqref{e:dyadic_unviscous}
\[
X_n(t)=
X_n(t_0)e^{-\int_{t_0}^t k_n X_{n+1}(s) ds}
+ \int_{t_0}^t k_{n-1}X_{n-1}^2(s)e^{-\int_{s}^t k_n X_{n+1}(\tau) d\tau} ds.
\] 
The thesis follows.
\end{proof}
\begin{prp}
If the initial condition of system~\eqref{e:dyadic_unviscous} has
infinitely-many non-negative components, then every solution is
Leray-Hopf.
\end{prp}\begin{proof}
Let $(n_i)_{i\geq1}$ be an increasing sequence such that $x_{n_i}\geq0$.
By Proposition~\ref{thm:positivity_mantained}, $X_{n_i}(t)\geq0$ for all
$t\geq0$, so that for all $i$, $E_{n_i-1}$ is a non-increasing
function. Since $E_{n_i-1}\uparrow\|X\|^2$ pointwise as
$i\rightarrow\infty$, $\|X \|^2$ is also non-increasing.
\end{proof}
We conclude this section by introducing the concept of a positive
solution.
\begin{defi}
We denote by $H^+$ the set of points in $H$ with all positive
components,
\[
H^+=\{x\in H:x_i>0\text{ for all }i\geq1\}.
\]
A positive solution is any solution such that $X(t)\in H^+$ for all
$t\geq0$. Of course positive solutions are always Leray-Hopf.
\end{defi}
It is a consequence of Proposition~\ref{thm:positivity_mantained} that
$H^+$ is closed for the dynamics, and actually a slighty stronger
result was proved in~\cite{BarFlaMor2011TAMS}: it is enough for the
initial condition to have a positive first component and all the other
components non-negative, to prove that $X(t)\in H^+$ for all $t>0$.

Positive solutions will turn out to have interesting regularizing
properties, in the next section we will prove that under very general
hypothesis, all solutions become positive after a finite time.

\section{Negative components}\label{s:negative}
The following statement shows that under very general hypothesis, all
solutions become positive after a finite time.

Only for this section we will weaken the hypothesis that the initial
condition belongs to $H$.

\begin{thm}\label{thm:all_positive_in_finite_time}
Let $x\in\R^{\N_+}$ be any initial condition with $x_1>0$ and let $X$ be a
weak solution. We suppose that one of the following hypothesis hold:
\begin{enumerate}
\item $x\in l^\infty$; there is an increasing sequence of indices
  $(n_i)_{i\geq1}$ such that $x_{n_i}\geq0$ for all $i\geq1$; for
  some $\delta\in(0,1)$, some $j_0\in\N$ and for all $j\geq j_0$, $n_{j+1}\leq
  k_{n_j}^{2/3\delta}$.
\item $x\in H$ and $X$ is a Leray-Hopf solution.
\end{enumerate}
Then there exists $\tau>0$ such that for all $t\geq \tau$ we have
$X(t)\in H^+$.
\end{thm}
We will need the Lemmas~\ref{lem:positivity_from_n_to_n+1}
and~\ref{lem:positivity_2nd_lemma} below. Both of them use the
following quantity,
\[
\eta_n:
=\sum_{k\leq\omega_n}x_k^2
=E_{\omega_n}(0),
\]
where $\omega_n=\inf\{i\geq n:x_{i+1}\geq0\}$. This quantity is useful
to bound $X_n$ and can be easily bounded itself, as we show presently.

Depending on which one of the two hypothesis of the theorem holds, it
may be either that $\omega_n<\infty$ and $x_{\omega_n+1}\geq0$, or
$\omega_n=\infty$ and $\eta_n=\|x\|^2$. By
applying~\eqref{eq:energy_least_comp_nonincr} or the Leray-Hopf
property, in both cases, for all $n\geq1$ and $t\geq0$,
\begin{equation}
\label{eq:X_n_bounded_by_eta}
X_n^2(t)\leq E_{\omega_n}(t)\leq\eta_n.
\end{equation}
In the case of Leray-Hopf solutions, $\eta_n=\|x\|^2$ for all $n$. In
the case of $l^\infty$ initial condition,
$\eta_n\leq\omega_n\|x\|_{l^\infty}^2$ and moreover, if $i$ is such
that $n_i\leq n<n_{i+1}$, then $\omega_n=n_{i+1}-1$, so
\begin{equation}\label{eq:omega_bounded_by_k_n}
\frac{\eta_n}{\|x\|_{l^\infty}^2}
\leq\omega_n
\leq k_{n_i}^{2/3\delta}
\leq k_n^{2/3\delta},
\end{equation}
definitively.
\begin{lem}\label{lem:positivity_from_n_to_n+1}
In the same hypothesis of
Theorem~\ref{thm:all_positive_in_finite_time}, if $X_n(t)\geq a>0$,
then $X_{n+1}(t+v_n(a))\geq0$, where
\[
v_n(a)=\frac{2^{2\beta}\eta_n^2+a^4}{k_{n+1}a^4\sqrt{\eta_n}}.
\]
Moreover, for all $a>0$
\begin{equation}\label{eq:v_n_summable}
\sum_{n=1}^\infty v_n(a)<\infty.
\end{equation}
\end{lem}
\begin{proof}
For all $x\in\R$, let $\tau_x:=\inf\{s\geq0:X_{n+1}(s)\geq x\}\in[0,\infty]$.  If
$X_{n+1}(t)\geq0$ there is nothing to prove, so we suppose $\tau_0>t$.

On $[t,\tau_0)$, $X_{n+1}\leq0$, so
  $X_n'=k_{n-1}X_{n-1}^2-k_nX_nX_{n+1}\geq0$ and hence $X_n\geq a$ on
  the same interval.

Let $x$ be any positive number (it will be fixed a few lines below). We
want to prove that
\begin{equation}\label{eq:bound_on_tau-x}
\tau_{-x}\leq t+\frac{\eta_n}{2k_na^2x}.
\end{equation}
If $X_{n+1}(t)>-x$, $\tau_{-x}\leq t$ and we are done. On the other
hand, if $\tau_{-x}>t$, on $[t,\tau_{-x})$, $X_{n+1}\leq -x$ and the
  energy of components from $n+1$ to $\omega_n$ decreases at least
  linearly, as we show presently.
  By~\eqref{eq:energy_least_comp_nonincr}, for all
  $s\in[t,\tau_{-x})$,
\[
E_n(s)=-\int_t^s2k_nX_n^2(u)X_{n+1}(u)du
\geq2(s-t)k_na^2x,
\]
so that by~\eqref{eq:X_n_bounded_by_eta}, we obtain
bound~\eqref{eq:bound_on_tau-x}:
\[
0\leq\sum_{i=n+1}^{\omega_n}X_i^2(s)
=E_{\omega_n}(s)-E_n(s)
\leq\eta_n-2k_na^2x(s-t).
\]

Now let $x=\frac{k_na^2}{2k_{n+1}\sqrt{\eta_n}}$. We claim that
with this choice of $x$, $X_{n+1}'\geq k_na^2/2$ on
$[\tau_{-x},\tau_0)$, yielding
\begin{equation}\label{eq:bound_on_tau0}
\tau_0-\tau_{-x}\leq\frac{2x}{k_na^2}
=\frac1{k_{n+1}\sqrt{\eta_n}}.
\end{equation}
We prove the claim,
\[
X_{n+1}'=k_nX_n^2-k_{n+1}X_{n+1}X_{n+2}
\geq k_na^2-k_{n+1}X_{n+1}X_{n+2}.
\]
Since $X_{n+1}\leq0$, if $x_{n+2}\geq0$ we conclude immediately that
$X_{n+1}'\geq k_na^2$. On the other hand, if $x_{n+2}<0$, then
$\omega_n\geq n+2$ and by~\eqref{eq:X_n_bounded_by_eta}
$X_{n+2}\geq-\sqrt{\eta_n}$ so that
\[
X_{n+1}'
\geq k_na^2+k_{n+1}\sqrt{\eta_n}X_{n+1}.
\]
Since $X_{n+1}(\tau_{-x})=-x$ and the RHS becomes negative only for
$X_{n+1}<-2x$, we see that $X_{n+1}\geq-x$ on
$[\tau_{-x},\tau_0)$ and finally
\[
X_{n+1}'
\geq k_na^2-k_{n+1}\sqrt{\eta_n}x
= k_na^2/2.
\]
Putting~\eqref{eq:bound_on_tau-x} and~\eqref{eq:bound_on_tau0}
together, we find $X_{n+1}(t+v_n(a))\geq0$ with
\[
v_n(a):=\frac{\eta_n}{2k_na^2x}+\frac1{k_{n+1}\sqrt{\eta_n}}
=\frac{2^{2\beta}\eta_n^2+a^4}{k_{n+1}a^4\sqrt{\eta_n}},
\]
so the first part of the statement is proved.

Finally, we turn to prove the convergence
of~\eqref{eq:v_n_summable}. This is obvious in the case of Leray-Hopf
solutions where $\eta_n$ does not depend on $n$. In the other case we
can reduce ourselves to prove
$\sum_{n=1}^\infty\frac{\eta_n^{3/2}}{k_n}<\infty$, which follows
from~\eqref{eq:omega_bounded_by_k_n}.
\end{proof}
\begin{lem}\label{lem:positivity_2nd_lemma}
In the same hypothesis of
Theorem~\ref{thm:all_positive_in_finite_time}, there exist two
summable sequences of positive numbers $(a_n)_{n\geq1}$,
$(s_n)_{n\geq1}$ depending only on $x$, such that for all $n\geq1$,
for all $t>0$ and for all $\varepsilon\in(0,1]$,
\begin{align}
\label{hyp:x_n_and_x_n+1_positive}
\text{if}\quad &X_n(t)\geq0,X_{n+1}(t)\geq0,\\
\label{eq:energy_transfer_lemma_7}
\text{then}\quad &E_n(t+\varepsilon^{-2}s_n)-E_{n-1}(t)\leq\varepsilon a_n.
\end{align}
\end{lem}
\begin{proof}
We follow the lines of the first part of Lemma~7
in~\cite{BarFlaMor2011TAMS}, weakening the hypothesis of positivity.

For $n\geq1$, let $s_n=\frac{2^{1+\beta}}{x_1}k_n^{-(1-\delta)/3}$, $a_n=Ck_n^{-(1-\delta)/3}$ and
\[
b_{n,\varepsilon}=2a_n^2s_nk_{n-1}\left(1+\frac2{k_{n+1}\sqrt{\eta_{n+2}}\varepsilon^{-2}s_n}\right)^{-1}.
\]
By using $\varepsilon\leq1$, $\eta_{n+2}\geq x_1^2$ and after some computations, one shows
that
\begin{multline*}
b_{n,\varepsilon}
\geq k_{n+2}^\delta\left(1+\frac2{k_{n+1}x_1s_n}\right)^{-1}K'C^2
= k_{n+2}^\delta\left(1+k_n^{-2/3-\delta/3}\right)^{-1}K'C^2\\
\geq k_{n+2}^\delta\left(1+k_1^{-2/3}\right)^{-1}K'C^2
\geq k_{n+2}^\delta KC^2
\geq KC^2,
\end{multline*}
where $K>0$ does not depend on $n$ or $\varepsilon$.

Thanks to this inequality, by a suitable choice of $C$, we can impose
that $b_{n,\varepsilon}\geq \eta_{n+2}^{3/2}$, for all $n\geq1$ and
$\varepsilon\in(0,1]$.

In fact, in the case of Leray-Hopf solutions we choose $C$ such that
$KC^2\geq\|x\|^3$, so for all $n\geq1$,  $\inf_{\varepsilon}
b_{n,\varepsilon}\geq\|x\|^3=\eta_{n+2}^{3/2}$. In the case of $l^\infty$
initial condition we choose $C$ such that
$KC^2\geq\|x\|_{l^\infty}^3$, so for all $n\geq1$,
$\inf_{\varepsilon}b_{n,\varepsilon}\geq\|x\|_{l^\infty}^3k_{n+2}^\delta$ and then
we apply~\eqref{eq:omega_bounded_by_k_n}.

Let $h=\varepsilon^{-2}s_n$ and fix $n$ and
$t$. By~\eqref{hyp:x_n_and_x_n+1_positive}, $E_{n-1}$ and $E_n$ are
both nonincreasing in $[t,t+h]$, so
\[
E_n(t+\varepsilon^{-2}s_n)-E_{n-1}(t)
\leq E_n(t+s)-E_{n-1}(t+s)
=X_n^2(t+s),
\]
for all $s\in[0,h]$. 

Now we proceed by contradiction. Suppose
that~\eqref{eq:energy_transfer_lemma_7} does not hold. Then the above
inequality implies that $X_n^2>\varepsilon a_n$ on $[t,t+h]$, yielding
\begin{multline}\label{eq:passaggi_secondo_lemma_positivita}
E_n(t+\varepsilon^{-2}s_n)-E_{n-1}(t)
=\int_0^hE_n'(t+s)ds+X_n^2(t)\\
=-\int_0^h2k_nX_n^2(t+s)X_{n+1}(t+s)ds+X_n^2(t)\\
\leq -2\varepsilon a_nk_n\int_0^hX_{n+1}(t+s)ds+X_n^2(t).
\end{multline}
We turn our attention to $X_{n+1}$. On the interval $[t,t+h]$ we have
\[
X_{n+1}'=k_nX_n^2-k_{n+1}X_{n+1}X_{n+2}
\geq \varepsilon a_nk_n-k_{n+1}\sqrt{\eta_{n+2}}X_{n+1}.
\]
We get
\[
X_{n+1}(t+s)\geq\varepsilon a_nk_n\int_0^se^{k_{n+1}\sqrt{\eta_{n+2}}(\tau-s)}d\tau.
\]
If this integral is computed and substituted
into~\eqref{eq:passaggi_secondo_lemma_positivita}, using the
inequality $e^{-x}-1+x\geq x\left(1+\frac2x\right)^{-1}$ one gets
\begin{multline*}
E_n(t+\varepsilon^{-2}s_n)-E_{n-1}(t)
\leq X_n^2(t)-\frac{2\varepsilon^2 a_n^2k_n^2}{k_{n+1}\sqrt{\eta_{n+2}}}\frac h{1+\frac2{k_{n+1}\sqrt{\eta_{n+2}}h}}\\
=X_n^2(t)-\frac{2a_n^2s_nk_{n-1}{\eta_{n+2}}^{-1/2}}{1+\frac2{k_{n+1}\sqrt{\eta_{n+2}}\varepsilon^{-2}s_n}}
=X_n^2(t)-b_{n,\varepsilon}{\eta_{n+2}}^{-1/2}
\leq X_n^2(t)-{\eta_{n+2}}\leq0.
\end{multline*}
Since we were pretending~\eqref{eq:energy_transfer_lemma_7} would not
hold, this is a contradiction.
\end{proof}

\begin{proof}[Proof of Theorem~\ref{thm:all_positive_in_finite_time}]
Let $(a_n)_{n\geq1}$ and
$(s_n)_{n\geq1}$ be as in Lemma~\ref{lem:positivity_2nd_lemma} and let
\[
\gamma:=\sqrt{x_1^2-\varepsilon\sum_{i=1}^\infty a_i},
\]
with $\varepsilon>0$ small enough to make the radicand
positive.
Let $t_0=0$. By Lemma~\ref{lem:positivity_from_n_to_n+1} there exist a
time $t_1$ such that $X_2(t_1)\geq0$. For $n\geq1$ let
\[
t_{n+1}=t_n+\varepsilon^{-2}s_n+v_{n+1}(\gamma),
\]
where $v_n$ is given again by Lemma~\ref{lem:positivity_from_n_to_n+1}.
Notice that $\sum_{k=1}^\infty s_k$ and $\sum_{k=1}^\infty
v_k(\gamma)$ converge by the two lemmas, hence $t_k\uparrow
\sup_kt_k=:\tau<\infty$.

We will show by induction that $X_n(t_{n-1})\geq0$ for all $n\geq1$,
and this will prove the positivity of all components at time $\tau$.

The cases $n=1,2$ are already done. Now, let $n\geq2$, we aim to prove
that $X_{n+1}(t_n)\geq0$. By inductive hypothesis, for $k=1,2,\dots,n$
we know $X_k(t_{k-1})\geq0$, meaning in particular that $X_k(t)\geq0$
and for all $t\geq t_{n-1}$. By Lemma~\ref{lem:positivity_2nd_lemma}
applied with $t=t_k$,
\[
E_k(t_k+\varepsilon^{-2}s_k)-E_{k-1}(t_k)\leq\varepsilon a_k,
\qquad k=1,2,\dots,n-1.
\]
Using the fact that when $X_{k+1}\geq0$, $E_k$ is nonincreasing we get
\[
E_k(t_{k+1})-E_{k-1}(t_k)
\leq\varepsilon a_k,
\qquad k=1,2,\dots,n-1.
\]
By summing over $k$ the latter or the former inequalities, we obtain
\begin{equation}\label{eq:En_unif_bdd}
E_{n-1}(t_{n-1}+\varepsilon^{-2}s_{n-1})\leq\sum_{k=1}^\infty\varepsilon a_k,
\end{equation}
yielding
\[
X_n^2(t_n-v_n(\gamma))
\geq E_n(t_n-v_n(\gamma))-\varepsilon\sum_{k=1}^\infty a_k.
\]
Either $X_{n+1}(t_n)\geq0$ (and we are done), or $X_{n+1}<0$ and $E_n$
is nondecreasing on $[0,t_n]$, giving
\[
X_n^2(t_n-v_n(\gamma))
\geq E_n(0)-\varepsilon\sum_{k=1}^\infty a_k
\geq x_1^2-\varepsilon\sum_{k=1}^\infty a_k
=\gamma^2.
\]
$X_n$ is nonnegative since $t_{n-1}$, so we can resolve the sign
yielding $X_n(t_n-v_n(\gamma))\geq\gamma$.
Lemma~\ref{lem:positivity_from_n_to_n+1} applies and we conclude
$X_{n+1}(t_n)\geq0$.

We still have to prove that $X(t)\in l^2$ for $t\geq\tau$, but this is
an easy consequence of positivity: we know that for $t\geq t_{n-1}$,
$X_n(t)\geq0$ and hence $E_{n-1}'(t)\leq0$, so, from
inequality~\eqref{eq:En_unif_bdd}, we deduce
\[
E_{n-1}(\tau)
\leq E_{n-1}(t_{n-1}+\varepsilon^{-2}s_{n-1})
\leq\sum_{k=1}^\infty \varepsilon a_k.
\]
Letting $n$ go to infinity we get the result.
\end{proof}

\section{Regularity of positive solutions}\label{s:regularity}
When only solutions with all positive components are considered, we
know from~\cite{BarFlaMor2011TAMS} that even if the initial condition
is in $l^\infty$, all solutions become $l^2$ immediately and are
Leray-Hopf from that point onward.

Even if the solution is very regular at some time, energy can be
conserved only for some finite time-interval and then anomalous
dissipation starts. This phenomenon implies that some regularity was
lost, in particular it means that for all $\epsilon >0$ the quantities
\[
\sup_n k_n^{\frac{1}{3}+\epsilon} X_n(t)
\qquad\text{and}\qquad
\sum_n k_n^{\frac{1}{3}} X_n(t),
\]
become infinite in finite time, even if they were finite for $t=0$.

The results of this and the following sections will prove instead that
some regularity is still mantained, specifically,
\begin{enumerate}
\item for a.e.\ $t>0$, and all $\alpha>\frac13$,  $\sup_{n} n^{-\alpha}k_n^{\frac{1}{3}} X_n(t)<\infty$;\label{en:first}
\item for a.e.\ $t>0$, and all $\alpha>\frac43$,  $\sum_n n^{-\alpha}k_n^{\frac{1}{3}} X_n(t)<\infty$;\label{en:second}
\item $\sup_{n} k_n^{\frac{1}{3}-\frac1{3\beta}} X_n(t)\leq C t^{-1/3}$ for all $t>0$.\label{en:third}
\end{enumerate}
Statement~\ref{en:second} follows trivially from
statement~\ref{en:first}, which in turn is a consequence of
Theorem~\ref{thm_reg1}, proved in Corollary~\ref{cor_dia_pos1}.
Statement~\ref{en:third} follows from uniqueness
(Section~\ref{sec:uniqueness}) and Theorem~\ref{thm_inv_reg}, and is
proved in Theorem~\ref{cor_limit_X_3}.

\begin{thm}\label{thm_reg1}
Let $x=(x)_{n\geq1}\in l^2$ with $x_n\geq0$ for all $n$; let
$X=(X)_{n\geq1}$ be a weak solution of~\eqref{e:dyadic_unviscous} with
initial condition $x$. Then there exists a constant $c$ depending only
$||x||_{l^2}$ and $\beta$ such that for any positive, non-increasing
sequence $(a_n)_{n\geq1}$ the following inequality holds
\begin{equation}\label{eq_thm_stima}
\mathcal L\{t>0|X_n(t)>a_n\text{ for some }n\}
\leq c \sum_n\frac{1}{k_na_n^3},
\end{equation}
where $\mathcal L$ denotes the Lebesgue measure.

The quantity $c=2^{7+\beta}||x||^2_{l^2}$ satisfies this theorem.
\end{thm}

The sequences $(a_n)_{n\geq1}$ which are meaningful for this statement
are those for which the sum on the right-hand side is finite.  The
example which in particular is important for us is $a_n:=M
k_n^{-\frac{1-\epsilon}{3}}$ with $M$ large, for in that case one
proves that the bound $X_n(t)\leq M k_n^{-\frac{1-\epsilon}{3}}$ holds
for all $t$ except for a set of small measure $\frac{c'||x||^2}{M^3}$
and all $n\geq1$.

\begin{proof} We suppose that the sum $\sum_n\frac{1}{k_na_n^3}$ is finite.

Let $I:=\cup_n \{t\in(0,+\infty)|X_n(t)>a_n\}$. The set $I$ is open,
so it is possible to approximate it from inside with finite unions of
intervals. To prove inequality \eqref{eq_thm_stima} it will be
sufficient to show that $\mathcal L(J)\leq c \sum_n\frac{1}{k_na_n^3}$
for any possible set  $J\subseteq I$ which is the finite union of intervals.
Let $J$ be the union of a finite number of disjoint intervals
\[
J=\cup_{k=1}^m [b_k,c_k),
\]
with $J\subseteq I$. Since we changed $I$ with $J$, we enjoy the
following property: for all $M>0$, the set $J\setminus [0,M)$ either
  is empty, or it has a minimum.  

Let $T:=\sup(J)+\sup_n\{\frac{3}{k_n\cdot a_n}\}$ (a time large enough
to be sure that definition \eqref{eq_def_ti} below always yields
$t_i<T$).

We are going to define a family of intervals $[s_i,t_i)$ whose union
  covers $J$ and such that their total measure is less than $c
  \sum_n\frac{1}{k_na_n^3}$. To each interval we will associate one
  component $n_i$ of the solution, in such a way that we can control
  $X_{n_i}$ on the interval. The variables $s_i$, $t_i$ and $n_i$ will
  be defined by transfinite induction on $t_i$, starting from $t_0=0$.
  Notice that the definitions will ensure that $s_i\leq t_i \leq T$
  for all ordinals $i$, with $s_i=t_i$ if and only if
  $s_i=t_i=T$. Moreover, if $i$ and $j$ are ordinals with $i<j$, then
  $t_i\leq s_j$.

We now give the definition by transfinite induction. Let $i$ be an
ordinal and suppose we already defined $t_j$ for all $j<i$. Let
$J_i:=\{t\in J| t\geq t_j \ \forall j<i \}$, if $J_i$ is empty we
define $s_i=T$ and $t_i=T$, otherwise $J_i$ has a minimum and we can define
\begin{gather*}
s_i:=\min(J_i),\\
n_i:=\min\{n\geq1|X_n(s_i)>a_n, X_n(s_i)\geq X_{n+2}(s_i)\}.
\end{gather*}
The fact that $n_i$ is well-defined follows by $s_i\in I$, $x\in l^2$
and $(a_n)_n$ non-increasing. The most subtle step in the proof is the definition of $t_i$ below:
\begin{equation}\label{eq_def_ti}
t_{i}:=\min \left\{
\begin{array}{l}
\inf\{t>s_i|X_{n_{i}}(t)<\frac{1}{2}X_{n_{i}}(s_i)\},\\
\inf\{t>s_i|X_{n_{i}+2}(t)>2X_{n_{i}}(s_i)\},\\
s_i+\frac{2}{k_{n_{i}+1}\cdot a_{n_{i}}}
\end{array}
\right\}.
\end{equation}
The latter ensures that $t_{i}-s_i\leq\frac{2}{k_{n_{i}+1}\cdot a_{n_{i}}}$.

We now consider in separate cases which of the quantities attains the
minimum in \eqref{eq_def_ti}:
\begin{align} \label{eq_primo_c}
\text{if}&\quad X_{n_{i}}(t_{i})=\frac{1}{2}X_{n_{i}}(s_i)
&&\text{then}\quad 
E_{n_{i}}(s_{i})-E_{n_{i}}(t_{i}) \geq\frac{3}{4} a_{n_{i}}^2;\\
\label{eq_secondo_c}
\text{if}&\quad X_{n_{i}+2}(t_{i})=2X_{n_{i}+2}(s_i)
&&\text{then}\quad 
E_{n_{i}+1}(s_{i})-E_{n_{i}+1}(t_{i}) \geq3 a_{n_{i}}^2;\\
\label{eq_terzo_c}
\text{if}&\quad t_{i}=s_i+\frac{2}{k_{n_{i}+1}\cdot a_{n_{i}}}
&&\text{then}\quad 
E_{n_{i}}(s_{i})-E_{n_{i}}(t_{i}) \geq 2^{-5-2\beta}
a_{n_{i}}^2.
\end{align}
The first and the second one are immediate, we prove the latter.  If
$t_{i}=s_i+\frac{2}{k_{n_i}\cdot a_{n_i}}$ then for all
$s\in(s_i,t_{i})$ we have $X_{n_i}(s)\geq\frac{1}{2}X_{n_i}(s_i)$ and
$X_{n_i+2}(s)\leq2X_{n_i+2}(s_i)\leq2X_{n_i}(s_i)$, yielding
\begin{align*}
X'_{n_i+1}(s)&=
k_{n_i}X^2_{n_i}(s)
-k_{n_i+1}X_{n_i+1}(s)X_{{n_i}+2}(s)\\
&\geq \frac14k_{n_i}X_{n_i}^2(s_i) -2
k_{n_i+1}X_{n_i+1}(s)X_{n_i}(s_i).
\end{align*}
Since $X_{n_i+1}(s_i)\geq 0$, we have (for all $t>s_i$)
\begin{align*}
X_{n_i+1}(t)&\geq
\int_{s_i}^t \frac14k_{n_i}X^2_{n_i}(s_i) 
e^{- 2k_{n_i+1}X_{n_i}(s_i)\cdot(t-s)}ds\\
&\geq
\frac{X_{n_i}(s_i)}{2^{3+\beta}}
\left(
1-e^{-2k_{n_i+1}X_{n_i}(s_i)\cdot(t-s_i)}
\right).
\end{align*}
This inequality allows to lower bound 
$E_{n_i}(s_i)-E_{n_i}(t_{i})$ as follows
\begin{align*}
E_{n_i}(s_i)-E_{n_i}(t_{i})&=
\int_{s_i}^{t_{i}} 2k_{n_i}X^2_{n_i}(t)X_{n_i+1}(t) dt\\
&\geq
\frac{k_{n_i}X^3_{n_i}(s_i)}{2^{4+\beta}}
\int_{s_i}^{t_{i}}
(1-e^{-2k_{n_i+1}X_{n_i}(s_i)\cdot(t-s_i)}) dt\\
& \geq
\frac{k_{n_i}X^3_{n_i}(s_i)}{2^{4+\beta}} 
\cdot \frac{1}{4} \cdot \frac{2}{k_{n_i+1}a_n}
\geq  \frac{X^2_{n_i}(s_i)}{2^{5+2\beta}}.
\end{align*}
This proves inequality \eqref{eq_terzo_c}.

Since $E_n$ is non-increasing in $t$ and $E_n(0)\leq||x||_{l^2}^2$,
from inequalities \eqref{eq_primo_c}, \eqref{eq_secondo_c} and
\eqref{eq_terzo_c}, we deduce that for all $n\geq1$
\[
\sharp \{i|n_i=n\}
\leq \frac{||x||_{l^2}^2}{\frac{3}{4} a_n^2}+ \frac{||x||_{l^2}^2}{3 a_n^2}+ \frac{||x||_{l^2}^2}{2^{-5-2\beta} a_n^2}
\leq \frac{||x||_{l^2}^2}{a_n^2}2^{6+2\beta}.
\]
Finally we are able to sum the measure of all intervals $[s_i,t_i)$:
\begin{align*}
\sum_i (t_i-s_i)&=
\sum_{n\geq1} \sum_{\{i|n_{i}=n\}} (t_{i}-s_i)
\leq
\sum_{n\geq1} \frac{||x||_{l^2}^2}{a_n^2}2^{6+2\beta} \cdot \frac{2}{k_{n+1} a_{n}}\\
&\leq 2^{7+\beta} ||x||_{l^2}^2
\sum_n \frac{1}{k_{n}a^3_{n}}.\tag*{\qedhere}
\end{align*}
\end{proof}
\begin{cor}\label{cor_dia_pos1}
Let $x=(x_n)_{n\geq1}\in l^2$, with $x_n\geq0$ for all $n$.  Then for
all $\alpha>\frac13$ and for a.e.\ $t>0$, the following inequality
holds:
\[
\sup_n n^{-\alpha}k_n^{\frac13} X_n(t)<\infty.
\] 
\end{cor}
\begin{proof}
Simply apply Theorem~\ref{thm_reg1} to the sequence $a_n:=M
n^\alpha k_n^{-\frac13}$ and let $M$ go to infinity.
\end{proof}
The following is a more subtle consequence of
Theorem~\ref{thm_reg1} that will be needed to prove uniqueness, in
the next section.

\begin{cor}\label{cor_limit_X_N}
There exists a constant $c=c(\beta)$ such that, if
$x=(x_n)_{n\geq1}\in l^2$, $x_n\geq0$ for all $n$, the following
inequality holds for all $n\geq1$ and $M>0$.
\[
\mathcal L(X_n>M):=\mathcal L\{t\geq 0|X_n(t)>M\}\leq \frac{c
||x||^2_{l^2}}{k_n M^3}.
\] 
\end{cor}
\begin{proof}
Take  $L>M$ and define
\[
a_i:=
\begin{cases}
L\  &i<n\\
M\  &i\geq n.
\end{cases}
\]
Apply Theorem~\ref{thm_reg1} to get:
\begin{align*}
\mathcal L(X_n>M) & \leq
2^{8+\beta} ||x||^2_{l^2} \sum_i \frac{1}{k_ia_i^3}\\
 & \leq
2^{8+\beta} ||x||^2_{l^2} 
\left(
\sum_{i=1}^{n-1} \frac{1}{k_iL^3} +
\sum_{i=n}^{\infty} \frac{1}{k_iM^3}
\right).
\end{align*}
Taking the limit as $L$ goes to infinity, we get
\[
\mathcal L(X_n>M)
\leq 
\frac{c(\beta)||x||^2_{l^2}}{k_n M^3}\tag*{\qedhere}.
\] 
\end{proof}

\section{Uniqueness}
\label{sec:uniqueness}
The next step is to prove the uniqueness of solutions of system
\eqref{e:dyadic_unviscous} with positive initial condition $x\in l^2$,
for all $\beta>1$. (The case $\beta\leq1$ was already proved
in~\cite{BarFlaMor2010CRAS}.)
\begin{thm}\label{thm_unicita}
Let $x=(x_n)_{n\geq1}\in l^2$ with $x_n\geq 0$ for all $n$. For all
$\beta>1$ there exists a unique weak solution $X$ of
\eqref{e:dyadic_unviscous} with initial condition $x$.
\end{thm}
Before starting the proof, we need some estimate of $\int_0^T X^3_N
(t)dt$ for $N$ large. Let $N$ be an integer, $T>0$ and $c=c(\beta)$ be the
constant of Corollary~\ref{cor_limit_X_N}. From the corollary, we
deduce that for all $y\geq0$,
\[
\phi(y):=
\mathcal L\{t\in[0,T]|X^3_N(t)>y\}
\leq \min\left\{
\frac{c||x||^2_{l^2}}{k_Ny} , T
\right\}
\]
Observe that  $\phi(y)=0$ for all $y>||x||^3$, so that
\begin{multline*}
\int_0^T X^3_N (t)dt
=\int_0^\infty \phi(y) dy
=\int_0^{||x||^3} \phi(y) dy\\
\leq\int_0^{||x||^3}\min\left\{\frac{c||x||^2_{l^2}}{k_N y} , T\right\}dy
\leq \frac{c||x||^2}{k_N T}T+ \int_\frac{c||x||^2}{k_N
  T}^{||x||^3} \frac{c||x||^2_{l^2}}{k_N y} dy
\end{multline*}
where we supposed that $N$ is large enough that
$K_N T>c||x||_{l^2}^{-1} $. By integrating, we obtain the following lemma.
\begin{lem}\label{lem_intX3_bound}
Let $x=(x_n)_{n\geq1}\in l^2$ with $x_n\geq 0$ for all $n$. Let $T>0$,
let $c=c(\beta)$ be the constant from Corollary~\ref{cor_limit_X_N}. For
all $N\geq1$ such that $K_N T>c||x||_{l^2}^{-1} $ the following
inequality holds
\begin{equation}\label{eq_intX3_bound}
\int_0^T X^3_N (t)dt\leq
\frac{c(\beta)||x||^2}{k_N}
\left(
1+\log\left(\frac{||x||T}{c(\beta)}\right)+N\beta\log(2)
\right)
\end{equation}
\end{lem}
\noindent
This lemma shows that as $N$ goes to infinity, the quantity $\int_0^T
X^3_N (t)dt$ tends to zero at least as $\frac{N}{k_N}$. This is in
accordance with the simple bound
\[
\int_0^T X^2_N X_{N+1} (t)dt\leq \frac{||x||^2}{k_N}
\]
which is immediate by energy balance.
\begin{proof}[Proof of Theorem~\eqref{thm_unicita}]
The first part of the proof follows ideas in~\cite{BarFlaMor2010CRAS}.

Let $X$ and $Y$ be two solution of system~\eqref{e:dyadic_unviscous} with
initial condition $x$. For all $n\geq1$ let us define $Z_n$ and $W_n$:
\begin{equation}\label{eq_def_ZW}
\begin{cases}
Z_n:=Y_n-X_n\\
W_n:=Y_n+X_n
\end{cases}
\qquad \text{so that} \qquad 
\begin{cases}
Y_n=\frac{W_n+Z_n}{2}\\
X_n=\frac{W_n-Z_n}{2}
\end{cases}
\end{equation}
It is easy to verify that $Z$ satysfies the following system
\begin{equation*}
  \begin{cases}
    Z_n'
      = k_{n-1} Z_{n-1}W_{n-1}
        -k_n \frac{Z_n W_{n+1}+W_n Z_{n+1}}{2},\\
    Z_n(0)=0,
  \end{cases}
  \qquad n\geq1,\quad t\geq0
\end{equation*}
Since $Y_n(t)=X_n(t)$ if and only if $Z_n(t)=0$, it will be sufficient
to prove that $Z_n(t)\equiv0$ for all $t\geq0$ and $n\geq1$. Let
\[
\psi_N(t):=\sum_{n=1}^N \frac{Z_n^2}{2^{n}},
\]
The functions $\psi_N(t)$ are non-negative, non-decreasing in $N$ and
such that $\psi_N(0)=0$.  We will prove that for all $t>0$, $\lim_{N\to\infty}
\psi_N(t)=0$ concluding that $\psi_N(t)\equiv0$ and $Z_n(t)\equiv
0$. We start by computing the derivative of $\psi_N$
\[
\psi'_N=
-\frac{k_N Z_N Z_{N+1} W_N}{2^N}
-\sum_{n=1}^N \frac{k_n}{2^n} Z_n^2W_{n+1}.
\]
We observe that $W_n(t)=X_n(t)+Y_n(t)$ is non-negative and use
\eqref{eq_def_ZW} to get
\begin{align*}
\psi'_N & \leq 
-2^{-N}k_N Z_N Z_{N+1} W_N
\leq 2^{-N}k_N (Y_N^2 X_{N+1}+ X_N^2 Y_{N+1})\\
 & \leq 2^{-N}k_N (Y_N^3+ X_{N+1}^3+ X_N^3 +Y_{N+1}^3).
\end{align*}
Since $\psi_N(0)=0$ we deduce
\[
\psi_N(t)
\leq 2^{-N}k_N \int_0^t \left[Y_N^3(s)+ X_{N+1}^3(s)+ X_N^3(s) +Y_{N+1}^3(s)\right] ds
\]
We can now apply Lemma~\ref{lem_intX3_bound} to get
\[
\psi_N(t) \leq 2^{-N}k_N \frac{N}{k_N} c(\beta,t,||x||) 
\]
where $c$ is a constant not depending on $N$. 

We conclude that $\lim_{N\to\infty} \psi_N(t)=0$ for all $t>0$.
\end{proof}

\section{The invariant region}\label{s:invariantregion}
In this section we work in the hypothesis $\beta\geq1$ and
$\sup_nk_n^{\frac13-\frac1{3\beta}} x_n<\infty$.  We want to
prove that $\sup_nk_n^{\frac13-\frac1{3\beta}} X_n(t)$ remains
finite and uniformly bounded in $t$.  (We already know from
Corollary~\ref{cor_dia_pos1} that this quantity is finite for
a.e.\ $t$.)

It is natural to consider the following change of variable:
$Y_n(t):=k_n^{\frac13-\frac1{3\beta}} X_n(t)=2^{\frac{\beta-1}3n}X_n(t)$. From
\eqref{e:dyadic_unviscous} we obtain that $(Y_n)_{n\geq1}$ solves

\begin{equation}\label{e:Y_system}
  \begin{cases}
    Y'_n
      = 2^{\frac{2\beta+1}{3}n-\frac{\beta+2}{3}} 
        \left(
        Y_{n-1}^2
        -2 Y_n Y_{n+1}
        \right),\\
    Y_n(0)=y_n,
  \end{cases}
  \qquad n\geq1,\quad t\geq0
\end{equation}
with $y_n:=2^{\frac{\beta-1}{3}n} x_n$.

For technical reasons we consider a finite dimensional (truncated) version
of the equations for $Y$. For every $N\geq1$ let $(Y_n^{(N)})_{1\leq n\leq N}$
be the solution to
\begin{equation}\label{e:Y_finite_system}
  \begin{cases}
    \frac d{dt} Y_n^{(N)}
      = 2^{\frac{2\beta+1}{3}n-\frac{\beta+2}{3}} 
        \left[
        \left(Y^{(N)}_{n-1}\right)^2
        -2 Y^{(N)}_n Y^{(N)}_{n+1}\right],\\
    Y^{(N)}_n(0)=y_n,
  \end{cases}
\end{equation}
for $n=1,\dots,N$, where for the sake of simplicity we have set
$Y_0^{(N)}=0$ and $Y_{N+1}^{(N)} = Y_N^{(N)}$, so to avoid writing the
border equations in a different form.  Let us now introduce the region
$A$ of $\R^2$ that will be invariant for the vectors
$(Y_n^{(N)},Y_{n+1}^{(N)})$,
\[
  A := \{ (x,y)\in[0,1]^2 : h(x)\leq y\leq g(x) \},
\]
\begin{figure}[htbp]
\begin{center}
\includegraphics[width=10cm]{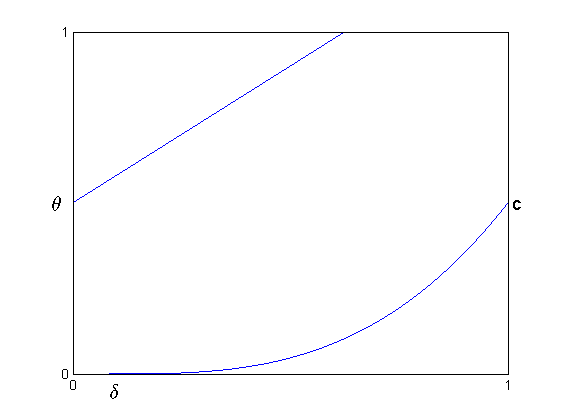}
\end{center}
\caption{Invariant Region} 
\end{figure}

where the functions $h$ and $g$ that provide the bounds of $A$
are defined as
\[
  g(x)
    = mx+\theta,
  \qquad\qquad
  h(x) =
      c\left(\frac{x-\delta}{1-\delta}\right)^3
\]
so that $g\leq1$ on $[0,\frac{1-\theta}m]$ and $h\geq0$ on $[\delta,1]$.

\begin{thm} \label{thm_inv_reg}
For $\delta=\frac{1}{12}$, $c=\frac{1}{2}$, $\theta=\frac{1}{2}$,
$m=\frac{4}{5}$ and for all $\beta\geq 1 $ the region $A$ is an
invariant region, that is if $(y_n,y_{n+1})$ belongs to $A$ for all
$n\leq N-1$ then $(Y^N_n(t), Y^N_{n+1}(t))$ belongs to $A$ for all
$n\leq N-1$ and $t\geq0$.
\end{thm}
\begin{proof}
For notation semplicity we drop the superscript $N$ along the proof.
We follow and improve ideas from~\cite{BarMorRom}. Consider the set
\[
B:=\{z\in[0,1]^N:(z_n,z_{n+1})\in A, n=1,2,\dots,N-1\}
\]
We want to prove that $Y\in B$ for all times, knowing that this is
true at time zero. This will follow from the fact that $\frac
d{dt}Y$ points inward on all the border of $B$.

By the definition of $B$, a point $Y=(Y_1,\dots,Y_N)$ is on the border
if (at least) one of the following constraints holds with equality
\begin{enumerate}
\item $h(Y_n)\leq Y_{n+1}$, for all $n\leq N-1$;
\item $Y_{n+1}\leq g(Y_n)$, for all $n\leq N-1$;
\item $0\leq Y_n\leq1$, for all $n\leq N$.
\end{enumerate}

For the first two cases, we must check that for $n=1,2,\dots,N-1$, the
derivative in time of the vector $(Y_n,Y_{n+1})$ points inward on the
border of $A$ when $Y\in B$. Steps 1 and 2 below deal with these cases.

In the third case, the positivity is trivial by
Proposition~\ref{thm:positivity_mantained}, while the other constraint
can be easily checked in dimension 1, as we do in step 3.

\smallskip\emph{Step 1. On the constraint $Y_{n+1}=h(Y_n)$.}

\noindent Let us start studing the border `along $h$' that is $Y_n\in[\delta,1]$
and $Y_{n+1}=h(Y_n)$. The inward normal is given by
$(-h'(Y_n),1)$.  From equation~\eqref{e:Y_finite_system} we obtain
\[
\left(
\begin{array}{c}
{Y}'_n\\
{Y}'_{n+1}\\
\end{array}
\right)
\sim
\left(
\begin{array}{c}
Y^2_{n-1}-2 Y_nY_{n+1}\\
2^{\frac{2\beta+1}{3}}(Y^2_{n}-2 Y_{n+1}Y_{n+2})\\
\end{array}
\right)
\]
where $\sim$ means that the two vectors have the same direction and
sense.

Consider the scalar product:
\begin{multline} \label{eq_scal_prod_h}
\left\langle(-h'(Y_n),1),\ ({Y}'_n,{Y}'_{n+1})\right\rangle =\\
-h'(Y_n)\cdot \left(Y^2_{n-1}-2 Y_nY_{n+1} \right)
 +2^{\frac{2\beta+1}{3}}(Y^2_{n}-2 Y_{n+1}Y_{n+2})\\
\geq -h'(Y_n)\cdot \left(1-2 Y_nh(Y_n) \right)
 +2^{\frac{2\beta+1}{3}}(Y^2_{n}-2 h(Y_n)g(h(Y_n)))
\end{multline}
Where the inequality comes from the fact that $(Y_1,\dots,Y_N)\in B$
and $Y_{N+1}=Y_N$.

Since we want the right-hand side of~\eqref{eq_scal_prod_h} to be
positive for arbitrarily large $\beta$, first of all we must check
that $\Phi_1(x):=x^2-2h(x)g(h(x))$ is positive for all
$x\in[\delta,1]$ and this is not difficult to verify, since $\Phi_1$
is just a real polinomial of degree 6. (See it plotted in
Figure~\ref{fig:funz_di_controllo}-a.)

As $\Phi_1>0$, we shall lower bound the right-hand side
of~\eqref{eq_scal_prod_h} by setting $\beta=1$, yielding
\[
\left\langle(-h'(Y_n),1),\ ({Y}'_n,{Y}'_{n+1})\right\rangle 
\geq -h'(Y_n)\cdot \left(1-2Y_nh(Y_n) \right)+2\Phi_1(Y_n)
=:\Phi_2(Y_n)
\]
where again $\Phi_2$ is a real polinomial of degree 6 not depending
on $\beta$ and positive on all $[\delta,1]$. (See
 Figure~\ref{fig:funz_di_controllo}-b.)

\begin{figure}[htbp]\label{fig:funz_di_controllo}
\begin{center}
\includegraphics[width=\textwidth]{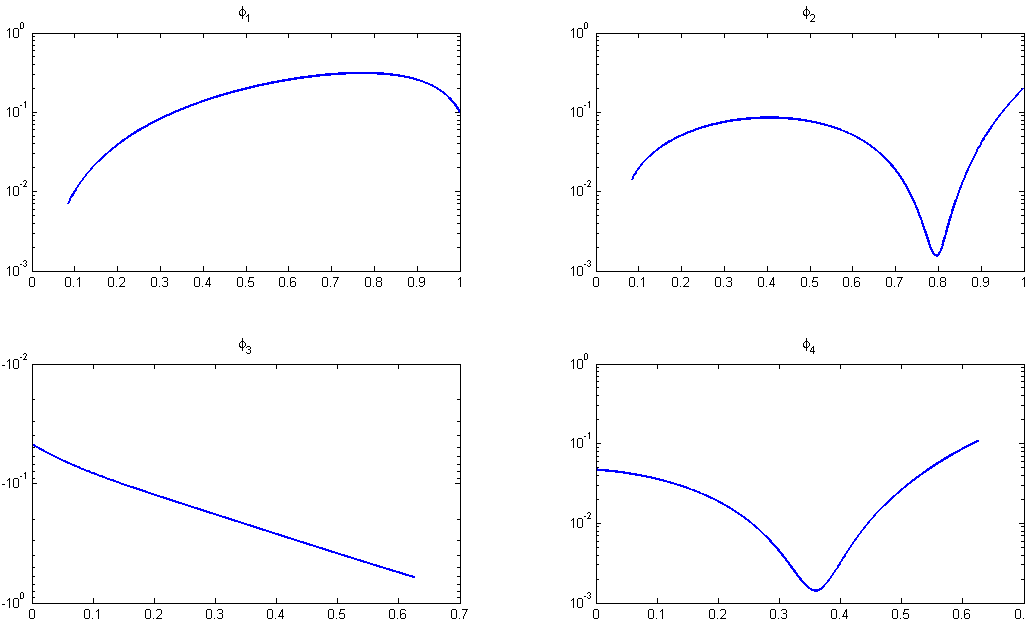}
\end{center}
\caption{Control functions} 
\end{figure}

\smallskip\emph{Step 2. On the constraint $Y_{n+1}=g(Y_n)$.}

\noindent Let us now turn to the border `along $g$' that is
$Y_n\in\left[0,\frac{1-\theta}{m}\right]$ and $Y_{n+1}=g(Y_n)$. The
inward normal is given by $(m,-1)$, whereas the scalar
product is given by:
\begin{multline} \label{eq_scal_prod_g}
\left\langle(m,-1),\ ({Y'}_n,{Y'}_{n+1})\right\rangle =\\
m\left(Y^2_{n-1}-2Y_nY_{n+1} \right)
-2^{\frac{2\beta+1}{3}}(Y^2_{n}-2Y_{n+1}Y_{n+2})\\
\geq m \left(0-2Y_ng(Y_n) \right)
-2^{\frac{2\beta+1}{3}}(Y^2_{n}-2g(Y_n)h(g(Y_n)))
\end{multline}
In analogy with the previous case, we define
$\Phi_3(x):=x^2-2g(x)h(g(x))$ and check that this degree 4 polynomial
is negative for $x\in\left[0,\frac{1-\theta}{m}\right]$. We can then
lower bound~\eqref{eq_scal_prod_g} by setting $\beta=1$, yielding
\[
\left\langle(m,-1),\ ({Y'}_n,{Y'}_{n+1})\right\rangle
\geq-2m Y_n g(Y_n) -2\Phi_3(Y_n)
=:\Phi_4(Y_n)
\]
One can check that $\Phi_4$ is a degree 4 polynomial positive on the
interval $\left[0,\frac{1-\theta}{m}\right]$.  (See
Figure~\ref{fig:funz_di_controllo}-c and
\ref{fig:funz_di_controllo}-d.)

\smallskip\emph{Step 3. On the constraint $Y_n=1$.}

\noindent We must prove that $Y_n'\leq0$ when $Y_n=1$ and $Y\in B$,
but these hypothesis imply that $Y_{n+1}\geq h(1)=c$, so that
\[
{Y'}_n
= 2^{\frac{2\beta+1}{3}n-\frac{\beta+2}{3}}\left( Y^2_{n-1}-2Y_nY_{n+1}\right)
\leq 2^{\frac{2\beta+1}{3}n-\frac{\beta+2}{3}}(1-2c)
=0
\]
This completes the proof.
\end{proof}

With standard techniques, making the limit on $N$ in the above theorem
it is possible to obtain the following:
\begin{cor} \label{cor_limit_X_2}
Let $\beta\geq 1$ and let $x=(x_n)_{n\geq1}$ be such that $x_n\geq0$
for all $n\geq1$ and $\sup_nk_n^{\frac13-\frac1{3\beta}}x_n<\infty$. If $X$ 
is the weak solution, then
\[
\sup_nk_n^{\frac13-\frac1{3\beta}}X_n(t) 
\leq 12\sup_nk_n^{\frac13-\frac1{3\beta}}x_n
\]
for all $t\geq 0$.
\end{cor}
\begin{proof}
Remember that equation~\eqref{eq_cambio_sca_2} defines a rescaling
which is still a solution of the original system. Let
$L:=\sup_nk_n^{\frac13-\frac1{3\beta}}x_n$ and $\delta=1/12$ as
in Theorem~\ref{thm_inv_reg}. Then $W(t):=\frac\delta LX(\frac\delta L
t)$ is the unique solution of system~\eqref{e:dyadic_unviscous} with
initial condition $\frac\delta Lx$. As in the beginning of this
section, define $Y=(Y_n)_{n\geq1}$, with
$Y_n:=k_n^{\frac13-\frac1{3\beta}}W_n$. Uniqueness for
system~\eqref{e:Y_system} follows from uniqueness for
system~\eqref{e:dyadic_unviscous}, so $Y$ is the unique solution of
the former, with initial condition $y=(y_n)_{n\geq1}$, defined by
$y_n=\frac\delta Lk_n^{\frac13-\frac1{3\beta}}x_n\in[0,\delta]$. Since
$[0,\delta]^2\subset A$, Theorem~\ref{thm_inv_reg} applies and thanks
to uniqueness, it is standard to prove that $Y^{(N)}$ converges to
$Y$, so we get $Y_n(t)\leq1$ uniformly in $n$ and $t$, which is what
we had to prove.
\end{proof}
\begin{thm}\label{cor_limit_X_3}
Let $X$ be the unique solution of system~\eqref{e:dyadic_unviscous}
with initial condition $x$ with non-negative components. Then there
exists a costant $c(\beta)>0$ such that the following inequality holds
for all $t>0$:
\[
\sup_nk_n^{\frac13-\frac1{3\beta}}X_n(t)\leq c(\beta)\|x\|^{\frac{2}{3}}t^{-1/3}.
\]
\end{thm}
\begin{proof}
Let $\psi(t):=\sup_nk_n^{\frac13-\frac1{3\beta}}X_n(t)$.  By
Corollary~\ref{cor_limit_X_2} we have,
\[
\psi(t)\leq12\psi(s)\qquad \forall \ s\in[0,t].
\]
To complete the proof it is sufficient to fix $c(\beta)$ in such a way
that there exists $s\in[0,t]$ such that
\begin{equation} \label{eq_bound_last_theorem}
12\psi(s)
\leq c(\beta)\|x\|^{\frac{2}{3}}t^{-1/3}
\end{equation}
Let us consider Theorem~\ref{thm_reg1}. Letting $a_n=M/r_n$, the thesis rewrites
\[
\mathcal L\{s>0|\sup_{n\geq1}\{r_nX_n(s)\}>M\}
\leq 2^{8+\beta} \|x\|^2\sum_{n\geq1}\frac{r_n^3}{k_nM^3}\, ,
\]
Now let
\begin{align*}
r_n&=12k_n^{\frac13-\frac1{3\beta}},
&
M&=c(\beta)\|x\|^{\frac{2}{3}}t^{-1/3},
&
c(\beta)>12\cdot2^{\frac{8+\beta}{3}},
\end{align*}
we get
\[
\mathcal L\Bigl\{s>0\Bigl|12\psi(s)>c(\beta)\|x\|^{\frac{2}{3}}t^{-1/3}\Bigr.\Bigr\}
\leq 12^3\frac{2^{8+\beta}}{c(\beta)^3}t
<t,
\]
hence there exists $s\in[0,t]$ such that
inequality~\eqref{eq_bound_last_theorem} holds.
\end{proof}


\bibliographystyle{plain}
\bibliography{dyadic}

\begin{thebibliography}{10}

\bibitem{BarFlaMor2010CRAS}
David Barbato, Franco Flandoli, and Francesco Morandin.
\newblock A theorem of uniqueness for an inviscid dyadic model.
\newblock {\em C. R. Math. Acad. Sci. Paris}, 348(9-10):525--528, 2010.

\bibitem{BarFlaMor2010PAMS}
David Barbato, Franco Flandoli, and Francesco Morandin.
\newblock Uniqueness for a stochastic inviscid dyadic model.
\newblock {\em Proc. Amer. Math. Soc.}, 138(7):2607--2617, 2010.

\bibitem{BarFlaMor2011AAP}
David Barbato, Franco Flandoli, and Francesco Morandin.
\newblock Anomalous dissipation in a stochastic inviscid dyadic model.
\newblock {\em Annals of Applied Probability}, 21(6):2424--2446, 2011.

\bibitem{BarFlaMor2011TAMS}
David Barbato, Franco Flandoli, and Francesco Morandin.
\newblock Energy dissipation and self-similar solutions for an unforced
  inviscid dyadic model.
\newblock {\em Trans. Amer. Math. Soc.}, 363(4):1925--1946, 2011.

\bibitem{BarMorRom}
David Barbato, Francesco Morandin, and Marco Romito.
\newblock Smooth solutions for the dyadic model.
\newblock {\em Nonlinearity}, 24(11):3083, 2011.

\bibitem{Bif}
L.~Biferale.
\newblock Shell models of energy cascade in turbulence.
\newblock {\em Annu. Rev. Fluid Mech.}, 35:441--468, 2003.

\bibitem{Che}
Alexey Cheskidov.
\newblock Blow-up in finite time for the dyadic model of the {N}avier-{S}tokes
  equations.
\newblock {\em Trans. Amer. Math. Soc.}, 360(10):5101--5120, 2008.

\bibitem{CheFriPav2010}
Alexey Cheskidov, Susan Friedlander, and Nata{\v{s}}a Pavlovi{\'c}.
\newblock An inviscid dyadic model of turbulence: the global attractor.
\newblock {\em Discrete Contin. Dyn. Syst.}, 26(3):781--794, 2010.

\bibitem{DesNov1974}
V.~N. {Desnianskii} and E.~A. {Novikov}.
\newblock Simulation of cascade processes in turbulent flows.
\newblock {\em Prikladnaia Matematika i Mekhanika}, 38:507--513, 1974.

\bibitem{EggGro}
Jens Eggers and Siegfried Grossman.
\newblock Anomalous turbulent velocity scaling from the navier-stokes equation.
\newblock {\em Physics Letters A}, 156:444--449, 1991.

\bibitem{KatPav}
Nets~Hawk Katz and Nata{\v{s}}a Pavlovi{\'c}.
\newblock Finite time blow-up for a dyadic model of the {E}uler equations.
\newblock {\em Trans. Amer. Math. Soc.}, 357(2):695--708 (electronic), 2005.

\bibitem{KisZla}
Alexander Kiselev and Andrej Zlato{\v{s}}.
\newblock On discrete models of the {E}uler equation.
\newblock {\em Int. Math. Res. Not.}, (38):2315--2339, 2005.

\bibitem{Wal}
Fabian Waleffe.
\newblock On some dyadic models of the {E}uler equations.
\newblock {\em Proc. Amer. Math. Soc.}, 134(10):2913--2922 (electronic), 2006.

\end{thebibliography}

\end{document}